\documentclass[11pt]{article}

\usepackage{amsmath,amsfonts,amssymb,amsthm}
\usepackage{fullpage,bbm,subfiles}

\newcommand{\cP}{\mathcal{P}}
\newcommand{\cQ}{\mathcal{Q}}

\newcommand{\K}{\mathbbm{K}}
\newcommand{\F}{\mathbbm{F}}

\newcommand{\one}{\mathbbm{1}}
\newcommand{\zero}{\emptyset}

\newcommand{\refeq}[1]{\overset{(\ref{#1})}{=}}

\newcommand{\true}[1]{\delta_{\{#1\}}}

\newcommand{\altS}{\hat{S}}
\newcommand{\altT}{\hat{T}}
\newcommand{\altU}{\hat{U}}
\newcommand{\altV}{\hat{V}}

\newtheorem{theorem}{Theorem}[section]
\newtheorem{corollary}[theorem]{Corollary}
\newtheorem{lemma}[theorem]{Lemma}

\newtheorem{problem}[theorem]{Problem}

\theoremstyle{definition}

\newcommand{\mydet}[1]{\det \left[ #1 \right]}

\newcommand{\ifno}[2]{
\makeatletter
\@ifundefined{#1}{ #2 }{}
}

\newcommand{\AND}{\qquad\text{and}\qquad}

\author{Adam W. Marcus \\ Princeton University}
\title{Discrete unitary invariance}

\begin{document}

\maketitle

\begin{abstract}

We show that certain determinantal functions of multiple matrices, when summed over the symmetries of the cube, decompose into functions of the original matrices.
These are shown to be true in complete generality; that is, no properties of the underlying vector space will be used apart from normal ring properties, and therefore hold in any commutative ring.
All proofs are elementary --- in fact, the majority are simply derivations.

\end{abstract}

\section{Introduction}

The goal of this paper is to show that certain functions which take multiple matrices as inputs can be reduced to a combination of unitarily invariant functions operating on disjoint subsets of the input matrices.
As such, the original functions will be invariant with respect to unitary conjugation among the subsets of input matrices.
We will achieve these functions by summing over a finite group of unitary matrices.
This group is guaranteed to exist whenever the underlying space is a commutative ring (simply by satisfying the ring conditions) and so all formulas will continue to hold when integrated against any measure which is invariant under the action of this group.
In particular, this will be true for any Haar measure defined over unitary matrices (where the property of being ``unitary'' is dictated by the underlying ring operations).
As such, they generalize results in \cite{mypolys, polys} that were used in the development of a finite version of free probability.
Formulas of this type were also used in \cite{IF4} in combination with the ``method of interlacing polynomials'' (developed in \cite{IF1, IF2}) to show the existence of Ramanujan graphs of all sizes and degrees.
These formulas can be viewed as quadratures formulas, polarization formulas, or statements about zonal spherical polynomials.

Will go out of our way to keep all proofs as elementary as possible.
The presentation is designed to highlight the generality of the results; in particular, it should be noted that nothing done in this article requires the existence of mulitiplicative inverses of any form.\footnote{
We will exercise one exception to the ``no dividing'' rule with respect to writing factorials --- rather than introduce additional notation (the Pochhammer symbol, for example), we will leave constants of the type $m!/n!$ where $m \geq n$ in a ``divided'' form.
In each such case, the resulting division will result in a positive integer and is used more as a notational convenience than as an actual scalar quantity (see the discussion at the end of Section~\ref{sec:rings}).}

\subsection{Organization}

In Section~\ref{sec:prelims}, we will introduce/review the notations, terminology and definitions we will use.
This will include defining two subgroups of matrices that we call $\cQ$ and $\cP$.
We also state our main technical lemma (Lemma~\ref{lem:minors}).
Section~\ref{sec:lemmas} contains our main lemmas regarding the symmetries of $\cQ$ and $\cP$ and then in Section~\ref{sec:apps}, we use these lemmas to prove identities involving some determinantal functions that exhibit these symmetries.
In Section~\ref{sec:cp}, we give applications of the previous sections to obtain formulas for characteristic polynomials of the type used in \cite{mypolys, polys}.
Finally, in Section~\ref{sec:end}, we conclude with an open problem.

\section{Definitions and such}\label{sec:prelims}

We begin with some notation:
As is common, we will write $[n]$ to denote the set $\{1, 2, \dots n \}$ and $\binom{S}{k}$ to denote the subsets of $S$ with size $k$.
We will write $\overline{S}$ to denote the complement of $S$ when the ambient set is clear (otherwise we will write $\Omega \setminus S$) and $S + T$ to denote the symmetric difference of $S$ and $T$ (the addition of their characteristic vectors in $\F_2$).
We will write $|S|$ to denote the cardinality of $S$ and 
\[
\|S\|_{1} = \sum_{s \in S} s
\]
for the $\ell^1$ norm.
It is easy to check that
\begin{equation}
\label{eq:sum1}
(-1)^{\|S\|_1 + \|T\|_1} = (-1)^{\|S + T\|_1}
\end{equation}
for any sets $S, T \subseteq [n]$, a fact that will be used throughout.
For $m \geq n$ and sets $T \subseteq [n]$ and $R = \{ R_1 < \dots < R_n \} \in \binom{[m]}{n}$, we define the induced set
\begin{equation}
\label{eq:setset}
T(R) = \{ R_{t} : t \in T \} \in \binom{R}{k} \subseteq \binom{[m]}{k}
\end{equation}
where the ordering of $R$ is taken to be the natural ordering of $[m]$. 
For example, given $R = \{ 1, 4, 7, 9 \}$ and $T = \{ 2, 3 \}$, we have $T(R) = \{ 4, 7 \}$.

\subsection{Rings}\label{sec:rings}

We will use $\K$ throughout to denote a commutative ring with multiplicative identity $\one$ and additive identity $\zero$.
As is customary, we will denote the additive inverse of $\one$ as $(-\one)$.
In the event that $\one = (-\one)$, we will call the ring {\em boolean}.
For an assertion $X$, we define the function $\true{X}$ as
\[
 \true{X}
=\begin{cases}
\one & \text{ if $X$ is true }  \\
\zero & \text{ if $X$ is false }.
\end{cases}
\]
We denote the ring of univariate polynomials over $\K$ as $\K[x]$ ($x$ being a formal variable).
% We will say that a matrix $M \in \K^{m \times n}$ is {\em binary} if all of its elements are from the set $\{ \zero, \one \}$ and {\em primitive} if all of its elements are from the set $\{ -\one, \zero, \one \}$.

Let $S_n$ be the symmetric group on $n$ elements.
The mapping $\sigma \mapsto P_\sigma$ where $P_\sigma$ is the $n \times n$ matrix with
\[
P_{\sigma}(i, j) 
= \true{\sigma(i) = j}
=
\begin{cases}
 \one & \text{if } \sigma(i) = j \\
 \zero & \text{otherwise}
\end{cases}
\]
is an injective homomorphism from $S_n$ into $\K^{n \times n}$; we will denote the image of this mapping as $\cP_n$ and refer to members of $\cP$ as {\em permutation matrices}.

For a set $S \subset [n]$, the mapping $S \mapsto Q_S$ where $Q_S$ is the $n \times n$ diagonal matrix with 
\[
 Q_S(i, i) 
 =
\begin{cases}
 -\one & \text{if } i \in S \\
 \one & \text{if } i \notin S \\
\end{cases}
\]
is an injective homomorphism from $\F_2^n$ into $\K^{n \times n}$; we call the image of this mapping $\cQ_n$, and refer to its members as {\em sign matrices}.

We remark that the only matrix in $\cP_n \cap \cQ_n$ is the identity matrix $I$, and that $\cQ_n  = \{ I \}$ if $\K$ is boolean.
Also, it should be clear (via the associated homomorphisms) that each element of $\cP_n$ and $\cQ_n$ is invertible.
One can check that the group generated by $\cP_n$ and $\cQ_n$ (via multiplication) is the largest subalgebra that can be guaranteed to be invertible (for general $\K$).
This group is isomorphic to the Coxeter group $B_n$ (or, dually, $C_n$), the group of symmetries of the $n$-cube (or, dually, the $n$-crosspolytope).
It is also known as the hyperoctahedral group.

Note that the only explicit ring elements we will use are $\{ (-\one, \zero, \one) \}$.
All other ring elements will appear implicitly as the elements of matrices.
If other constants appear (like positive integers), they are not to be considered ring elements.
Rather for an element $k \in \K$ and a nonnegative integer $n$, the following interpretations should be used
\begin{enumerate}
 \item $n k$: adding $n$ copies of $k$ (using ring addition)
 \item $k^n$: multiplying $n$ copies of $k$ (using ring multiplication).
\end{enumerate}
We direct the reader to \cite{lang} for definitions related to rings.

\subsection{Determinants}

Since $\K$ is commutative, we can take the usual definition of determinant for matrices in $\K^{n \times n}$:
\begin{equation}
\label{def:det}
\mydet{A} = \sum_{\sigma \in S_n} (-\one)^{|\sigma|} \prod_i A(i, \sigma(i))
\end{equation}
where $S_n$ denotes the symmetric group on $n$ elements and $|\sigma|$ denotes the parity of $\sigma$.
Note that all additions and multiplications are considered to be with respect to the ring operations.

For sets $S, T \in \binom{[n]}{k}$, we recall that the $S, T$ {\em minor} of $A$ is defined as
\[
[A]_{S, T} = 
\begin{cases}
\mydet{\{A(i,j)\}_{i\in S, j\in T}} & \text{for } k > 0 \\
1 & \text{for } k = 0.
\end{cases}
\]
Here, the notation 
\[
 \mydet{\{A(i,j)\}_{i\in S, j\in T}}
\]
is used to denote the determinant of the submatrix of $A$ with rows indexed by $S$ and columns indexed by $T$.
For $A \in \K^{n \times n}$, we also define
\begin{equation}
 \label{def:A^k}
 [A]^{(k)} = \sum_{S \in \binom{[n]}{k}} [A]_{S, S}
\end{equation}
(with $[A]^{(0)} = 1$).
We show in Lemma~\ref{lem:cp} that $[A]^{(k)}$ is equal to the coefficient of $x^{n-k}$ in the polynomial
\[
 \mydet{x I + A} \in \K[x]. 
\]
Thus, to the extent that it makes sense for a given $\K$ (the reals, for example), $[A]^{(k)}$ can be viewed as the $k$th elementary symmetric polynomial evaluated at the $n$ eigenvalues of $A$.

Our main technical tools will be the following decompositions:
\begin{lemma}
\label{lem:minors}
For $A \in \K^{m \times n}$ and $B \in \K^{n \times p}$, we have
\begin{equation}
\label{eq:mult}
[AB]_{S, T} = \sum_{U \in \binom{[n]}{k}} [A]_{S, U} [B]_{U, T}
\end{equation}
for all $S \in \binom{[m]}{k}$ and $T \in \binom{[p]}{k}$, and
\begin{equation}
\label{eq:add}
\mydet{A + B} = \sum_{S, T \subseteq [n]}(-\one)^{\|S + T\|_{1}} [A]_{S, T} [B]_{\overline{S}, \overline{T}}
\end{equation}
for all $A, B \in K^{n \times n}$.
\end{lemma}
\begin{proof}
The first identity is a direct application of the well known Binet--Cauchy theorem, while the second is merely a grouping of terms in (\ref{def:det}), see \cite{hornjohnson}.
\end{proof}

One technical issue that will arise concerning (\ref{eq:add}) is that the correctness of (\ref{eq:add}) depends upon the ground set being $[n]$ (otherwise the $\| S + T \|_{1}$ term would be altered).
In order to apply it to quantities like $[A + B]_{X, X}$ where $X$ is some subset of $[n]$, we will need to keep track of rows both in the frame of $X$ and in the frame of the larger matrix.
This motivates the introduction of the induced set constructs defined in (\ref{eq:setset}).

\section{Symmetries of $\cQ$ and $\cP$}\label{sec:lemmas}

In this section, we prove the main technical lemmas that we will use.
The first regards symmetries of $\cQ_n$ and the second the symmetries of $\cP_n$.

\begin{lemma}
\label{lem:cancel}
 Let $S, T, U, V \subseteq [n]$ with $|S| = |T|$ and $|U| = |V|$.
Then 
\begin{equation}
 \label{eq:cancel}
 \sum_{Q \in \cQ_n} [Q]_{S, T} [Q]_{U, V} = (\one + \one)^n \true{S = T = U = V}.
\end{equation}
\end{lemma}
\begin{proof}
We will refer to the matrices in $Q$ via the homomorphism from subsets of $[n]$.
Hence
\[
\sum_{Q \in \cQ_n} [Q]_{S, T} [Q]_{U, V} 
= \sum_{X \subseteq [n]} [Q_X]_{S, T} [Q_X]_{U, V}.
\]
We now claim that for fixed $X, U, V$, we have
\begin{equation}
\label{eq:fourier1}
 [Q_X]_{U, V} = 
 (-\one)^{|U \cap X|} \true{U=V}
\end{equation}
 To see this, note that $[Q_S]_{U, V} = \zero$ for $U \neq V$ since $Q_S$ is diagonal (the case where $U = V$ then follows directly from the definition).
 Hence it suffices to show
 \[
 \sum_{Q \in \cQ_n} [Q]_{U, U} [Q]_{V, V} = (\one + \one)^n \true{U=V}.
\]

Let $A$ be the $2^n \times 2^n$ matrix (indexed by subsets of $[n]$)
\[
 A(S, T) = (-\one)^{|S \cap T|}.
\]
It is easy to check that $A$ is an $n$-fold tensor product over the Fourier matrix 
\[
 F = 
\begin{pmatrix}
 \one & \one \\
 \one & -\one
\end{pmatrix}
\]
(where the multiplication in the tensor product is ring multiplication).
Since $A$ is symmetric, we have
\begin{equation}
 \label{eq:id}
A A^T
= A^T A  
= A^2 
= \otimes^n F^2
= \otimes^n \left( (\one + \one) I_2 \right)
= (\one + \one)^n I_{2^n}.
\end{equation}
On the other hand, by (\ref{eq:fourier1}), we have
\[
\sum_{X \subseteq [n]} [Q_X]_{U, U} [Q_X]_{V, V} 
= \sum_{X \subseteq [n]} (-\one)^{|X \cap U|} (-\one)^{|X \cap V|} 
= \sum_{X} A(X, U)A(X, V).
\]
and so by (\ref{eq:id}), 
\[
 \sum_{X \in [n]} [Q]_{U, U} [Q]_{V, V} = (\one + \one)^n I_{2^n}(U, V)
\]
which is exactly what was needed.
\end{proof}

\begin{lemma}
 \label{lem:cancel2}
 For $S, T, U \in \binom{[n]}{k}$, we have
\begin{equation}
 \label{eq:cancel2}
 \sum_{P \in \cP_n} [P]_{S, T} [P^{-1}]_{U, S} = k!(n-k)!\true{T=U}.
\end{equation}
\end{lemma}
\begin{proof}
 As in Lemma~\ref{lem:cancel}, we will refer to the matrices $P$ via their homomorphism with $S_n$, so that 
\[
\sum_{P \in \cP_n} [P]_{S, T} [P^{-1}]_{U, S} 
= \sum_{\pi \in S_n} [P_{\pi}]_{S, T} [P_{\pi^{-1}}]_{U, S} 
= \sum_{\pi \in S_n} [P_{\pi}]_{S, T} [P_{\pi}]_{S, U}.
\]
For $\pi \in S_n$, we will write $\pi(U)$ to denote the image of the elements of $U$ under $\pi$.
Note that for fixed $\pi, U, V$, we have $[P_\pi]_{U, V} = \zero$ whenever $\pi(U) \neq V$ (since there will be a row of $0's$).
Hence the sum will be zero whenever $U \neq T$.
When $U = T$, on the other hand, the sum is {\em still} zero unless $\pi(S) = T$, and there are $k!(n-k)!$ ways for that to happen.
In each such case, the product becomes a perfect square, and so is $\one$ regardless of the sign of the permutation.
\end{proof}

Note that the factor $k! (n-k)!$ should not be considered as an element in $\K$ (and so should not be multiplied using ring multiplication).
The term $k! (n-k)! \one$ should instead be interpreted as a sum of $k!(n-k)!$ copies of $\one$ (which {\em is} an element in $\K$).

\section{Applying symmetries to determinants}\label{sec:apps}

\begin{lemma}
\label{lem:symm}
Let $A \in \K^{n \times n}, B \in \K^{m \times n}, C \in K^{n \times r}$.
Then
\[
 \sum_{Q \in \cQ_n }  \sum_{P \in \cP_n} [B (Q P) A (Q P)^{-1} C ]_{X, Y} 
= (\one + \one)^n k!(n-k)! [BC]_{X, Y}  [A]^{(k)}.
\]
for all $X \in \binom{[m]}{k}$ and $Y \in \binom{[r]}{k}$.
\end{lemma}
\begin{proof}
By Lemma~\ref{lem:minors}, we have for fixed $Q \in \cQ_n$ and $P \in \cP_n$, 
\[
[B Q P A P^{-1} Q C]_{X, Y} 
 \refeq{eq:mult} \sum_{S, T, U, V \in \binom{[n]}{k}} [B]_{X, S} [Q]_{S, T} [PAP^{-1}]_{T, U} [Q]_{U, V} [C]_{V, Y}.
\]
Hence by Lemma~\ref{lem:cancel}
\begin{align*}
\sum_{Q \in \cQ_n } [B Q P A P^{-1} Q C]_{X, Y} 
 &= \sum_{S, T, U, V \in \binom{[n]}{k}} [B]_{X, S} [P A P^{-1}]_{T, U} [C]_{V, Y} (\one + \one)^n \true{S = T = U = V}
\\&= (\one + \one)^n \sum_{U \in \binom{[n]}{k}} [B]_{X, U} [P A P^{-1}]_{U, U} [C]_{U, Y}
\\& \refeq{eq:mult} (\one + \one)^n \sum_{S, T, U \binom{[n]}{k}} [B]_{X, U} [P]_{U, S} [A]_{S, T} [P^{-1}]_{T, U} [C]_{U, Y}.
\end{align*}
Now applying Lemma~\ref{lem:cancel2}, 
\begin{align*}
\sum_{Q \in \cQ_n }  \sum_{P \in \cP_n} [B Q P A P^{-1} Q C ]_{X, Y} 
 &= (\one + \one)^n k!(n-k)!\sum_{S, T, U \in \binom{[n]}{k}} [B]_{X, U} [A]_{S, T}  [C]_{U, Y}  \true{S = T}
\\&= (\one + \one)^n k!(n-k)! \sum_{S, U \in \binom{[k]}{i}} [B]_{X, U}  [A]_{S, S} [C]_{U, Y}
\\&= (\one + \one)^n k!(n-k)! [BC]_{X, Y}  [A]^{(k)}.
\end{align*}
where the last equality used both (\ref{def:A^k}) and (\ref{eq:mult}).
\end{proof}

%%%%%%%%%%%%%%%%%%%%%%%%

\begin{lemma}
\label{lem:asymm}
Let $A \in \K^{n \times m}, B \in \K^{p_1 \times n}, C \in K^{m \times r_1}, E \in \K^{m \times n}, F \in \K^{p_2 \times m}, G \in \K^{n \times r_2}$.
Then
\begin{align*}
 \sum_{Q \in \cQ_n }  \sum_{P \in \cP_n} \sum_{Q' \in \cQ_m }  &\sum_{P' \in \cP_m} [ B (Q P) A (Q'P')^{-1} C ]_{X, Y} [ F (Q' P') E (Q P)^{-1} G ]_{W, Z}
\\&= \true{j=k} (\one + \one)^{n+m} k!(n-k)!k!(m-k)! [BG]_{X, Z} [FC]_{W, Y} [AE]^{(k)}
\end{align*}
for all $X \in \binom{[p_1]}{k}, Y \in \binom{[r_1]}{k}$, $W \in \binom{[p_2]}{j}, Z \in \binom{[r_2]}{j}$.
\end{lemma}
\begin{proof}
 Just as in Lemma~\ref{lem:symm}, we separate out 
 \[
[ B Q P A {P'}^{-1} Q' C ]_{X, Y}  \sum_{S, T \in \binom{[n]}{k}} \sum_{U, V \in \binom{[m]}{k}} [B]_{X, S} [Q]_{S, T} [PA{P'}^{-1}]_{T, U} [Q']_{U, V} [C]_{V, Y}.
\]
and similarly 
\[
[ F Q' P' E P^{-1} Q G ]_{W, Z}  \sum_{\altS, \altT \in \binom{[n]}{j}} \sum_{\altU, \altV \in \binom{[m]}{j}} [F]_{W, \altU} [Q']_{\altU, \altV} [P' E P^{-1}]_{\altV, \altT} [Q]_{\altT, \altS} [G]_{\altS, Z}.
\]
By Lemma~\ref{lem:cancel}, we have 
\[
 \sum_{Q \in \cQ_n} \sum_{Q'\in \cQ_m} [Q]_{S, T}  [Q']_{U, V} [Q]_{\altT, \altS} [Q']_{\altU, \altV} = (\one + \one)^{n+m} \true{S = T = \altS = \altT}\true{U = V = \altU = \altV}
\]
where in order for the two $\delta$s to be satsified, we must have $j = k$.
Hence we get
\begin{align*}
\sum_{Q \in \cQ_n} \sum_{Q'\in \cQ_m } &[ B Q P A {P'}^{-1} Q' C ]_{X, Y} [ F Q' P' E P^{-1} Q' G ]_{W, Z}
\\&= \true{j=k} (\one + \one)^{n+m} \sum_{S \in \binom{[n]}{k}} \sum_{U \in \binom{[m]}{k}} [B]_{X, S} [PA{P'}^{-1}]_{S, U} [C]_{U, Y} [F]_{W, U} [P' E {P}^{-1}]_{U, S} [G]_{S, Z}.
\end{align*}
Now we can reduce
\[
[PA{P'}^{-1}]_{S, U} [P' E P^{-1}]_{U, S} 
= \sum_{T, \altT \in \binom{[n]}{k}} \sum_{V, \altV \in \binom{[m]}{k}} [P]_{S, T} [A]_{T, V} [{P'}^{-1}]_{V, U} [P']_{U, \altV} [E]_{\altV, \altT} [P^{-1}]_{\altT, S}
\]
and so by Lemma~\ref{lem:cancel2} 
\begin{align*}
\sum_{P, P'} [PA{P'}^{-1}]_{S, U} [P' E P^{-1}]_{U, S}
&= k!(n-k)!k!(m-k)! \sum_{T, \altT \in \binom{[n]}{k}} \sum_{V, \altV \in \binom{[m]}{k}} [A]_{T, V} [E]_{\altV, \altT} \true{V = \altV}\true{T = \altT}
\\ &= k!(n-k)!k!(m-k)! \sum_{T\in \binom{[n]}{k}} \sum_{V \in \binom{[m]}{k}} [A]_{T, V} [E]_{V, T}
\\ &\refeq{eq:mult} k!(n-k)!k!(m-k)! \sum_{T\in \binom{[n]}{k}}  [AE]_{T, T}
\\ &\refeq{def:A^k} k!(n-k)!k!(m-k)! [AE]^{(k)}.
\end{align*}
Combining the two, we get
\begin{align*}
 \sum_{Q, Q', P, P'} & [ B Q P A {P'}^{-1} Q' C ]_{X, Y} [ F Q' P' E P^{-1} Q G ]_{W, Z}
\\&= \true{j=k}(\one + \one)^{n+m} k!(n-k)!k!(m-k)! \sum_{S \in \binom{[n]}{k}} \sum_{U \in \binom{[m]}{k}} [B]_{X, S} [C]_{U, Y} [F]_{W, U} [G]_{S, Y} [AE]^{(k)}
\\&\refeq{eq:mult}\true{j=k} (\one + \one)^{n+m} k!(n-k)!k!(m-k)! [BG]_{X, Z} [FC]_{W, Y} [AE]^{(k)}
\end{align*}
as required.
\end{proof}

\section{Characteristic Polynomials}\label{sec:cp}

We start with the following simple lemma:
\begin{lemma}
\label{lem:cp}
 Let $A \in \K^{n \times n}$.
Then we have
\[
 \mydet{xI + A} 
= \sum_{i=0}^n x^{n-i} [A]^{(k)}
\]
as a polynomial in $\K[x]$.
\end{lemma}
\begin{proof}
 Using Lemma~\ref{lem:minors}, we have
\[
 \mydet{xI + A} 
= \sum_{k} \sum_{S, T \in \binom{[n]}{k}} (-\one)^{\| S + T \|_1} [A]_{S, T} [xI]_{\overline{S}, \overline{T}} 
\]
where, since $xI$ is diagonal, $[xI]_{S, T} = \zero$ unless $S = T$.
Furthemore, for a scalar $c$ (which $x$ essentially is) and matrix $X \in \K^{k \times k}$, it is easy to check that 
\[
 \mydet{cX} = c^k \mydet{X}.
\]
Substituting in gives the lemma.
\end{proof}

Combining Lemma~\ref{lem:cp} with the results in Section~\ref{sec:apps} gives the following corollaries:

\begin{corollary}
\label{cor:mult}
Let $A \in \K^{m \times m}$, $B \in \K^{n \times m}$ and $C \in \K^{m \times n}$ 
and let
\[
\mydet{xI + BC} 
= \sum_{i=0}^n x^{n-i} p_i
\AND
\mydet{xI + A} 
= \sum_{i=0}^n x^{m-i} q_i
\]
and
 \[
\sum_{P \in \cP_n} \sum_{Q \in \cQ_n} \mydet{xI +  B (Q P) A (QP)^{-1} C} 
=  \sum_{i=0}^n x^{n-i} r_i
\]
be polynomials in $\K[x]$.
Then 
\[
r_k = (\one + \one)^m k!(m-k)! p_k q_k
\]
for all $0 \leq k \leq n$.
\end{corollary}
\begin{proof}
Using Lemma~\ref{lem:cp}, the lemma is equivalent to showing 
\[
 \sum_{P \in \cP_m} \sum_{Q \in \cQ_m} [B Q P A P^{-1}Q C)]^{(k)}
= (\one + \one)^m k!(m-k)! [BC]^{(k)}[A]^{(k)}
\]
But this follows directly from Lemma~\ref{lem:symm}, as
\[
\sum_{P \in \cP_m} \sum_{Q \in \cQ_m} [B Q P A P^{-1} Q C)]^{(k)}
= \sum_{X \in \binom{[m]}{k}} (\one + \one)^m k!(m-k)! [BC]_{X, X}  [A]^{(k)}
\]
as required.
\end{proof}

\begin{corollary}
\label{cor:symmp}
Let $A, B \in \K^{n \times n}$ and let
\[
\mydet{xI + A} 
= \sum_{i=0}^n x^{n-i} p_i
\AND
\mydet{xI + B} 
= \sum_{i=0}^n x^{n-i} q_i.
\]
and
 \[
\sum_{P \in \cP_n} \sum_{Q \in \cQ_n} \mydet{xI + ( Q P A P^{-1} Q + B)} 
=  \sum_{i=0}^n x^{n-i} r_i.
\]
be polynomials in $\K[x]$.
Then 
\[
r_k = (\one + \one)^n\sum_{i = 0}^k \frac{(n-i)!(n-k+i)!}{(n-k)!} p_i q_{k-i}
\]
for all $0 \leq k \leq n$.
\end{corollary}
\begin{proof}
By Lemma~\ref{lem:cp}, it is again sufficient to show
\begin{equation}
 \label{eq:needforsymmp}
\sum_{P \in \cP_n} \sum_{Q \in \cQ_n}  [Q P A P^{-1} Q + B]^{(k)} = (\one + \one)^n \sum_{i = 0}^k \frac{(n-i)!(n-k+i)!}{(n-k)!} [A]^{(i)} [B]^{(k-i)}.
\end{equation}
By Lemma~\ref{lem:minors}, we have 
\begin{align*}
[Q P A P^{-1} Q + B]^{(k)} 
&= \sum_{X \subset \binom{[n]}{k}} [Q P A P^{-1} Q + B]_{X, X}
\\&= \sum_{X \subset \binom{[n]}{k}} \sum_{i=0}^k \sum_{S, T \in \binom{[k]}{i}}[Q P A P^{-1} Q]_{S(X), T(X)}[B]_{\overline{S}(X), \overline{T}(X)}
\end{align*}
where by Lemma~\ref{lem:symm}, we have 
\[
 \sum_{P \in \cP_n} \sum_{Q \in \cQ_n} [Q P A P^{-1} Q]_{S(X), T(X)} = (\one + \one)^n i!(n-i)! [I]_{S(X), T(X)}  [A]^{(i)}.
\]
Since $[I]_{S(X), T(X)} = \true{S = T}$, we can substitute to get
\begin{equation}
 \label{eq:halfwaysymm}
\sum_{P \in \cP_n} \sum_{Q \in \cQ_n}  [Q P A P^{-1} Q + B]^{(k)} = (\one + \one)^n \sum_{i = 0}^k i!(n-i)! \sum_{X \subset \binom{[n]}{k}} \sum_{S \in \binom{[k]}{i}} [A]^{(i)} [B]_{\overline{S}(X), \overline{S}(X)}.
\end{equation}
Now we can write
\[
 \sum_{X \subset \binom{[n]}{k}} \sum_{S \in \binom{[k]}{i}} [B]_{\overline{S}(X), \overline{S}(X)}
= \sum_{X \subset \binom{[n]}{k}} \sum_{T \in \binom{X}{k-i}} [B]_{T, T}
= \sum_{X \subset \binom{[n]}{k}} \sum_{T \in \binom{[n]}{k-i}} [B]_{T, T} \true{T \subseteq X}
\]
and swap the order of summation to get
\[
\sum_{X \subset \binom{[n]}{k}} \sum_{S \in \binom{[k]}{i}} [B]_{\overline{S}(X), \overline{S}(X)}
= \binom{n-k+i}{i} \sum_{T \in \binom{[n]}{k-i}} [B]_{T, T}
\refeq{def:A^k} \binom{n-k+i}{i} [B]^{(k-i)}.
\]
Substituting this into (\ref{eq:halfwaysymm}) then gives (\ref{eq:needforsymmp}) as required.
\end{proof}

\begin{corollary}
 Let $A, B \in \K^{n \times m}$ and $C, D \in \K^{m \times n}$ with $n \leq m$.
For $Q, P \in \K^{n \times n}$ and $Q', P' \in \K^{m \times m}$, define the matrices
\[
 M(Q, Q', P, P') = \big( (Q P) A (Q'P')^{-1} + B\big) \big( (Q' P') C (QP)^{-1} + D \big)
\]
and let
\[
\mydet{xI + AC} 
= \sum_{i=0}^n x^{n-i} (-\one)^i p_i
\AND
\mydet{xI + BD} 
= \sum_{i=0}^n x^{n-i} q_i.
\]
and
 \[
\sum_{P, P'} \sum_{Q, Q'} \mydet{xI + M(Q, Q', P, P')} 
=  \sum_{i=0}^n x^{n-i} r_i.
\]
be polynomials in $\K[x]$.
Then 
\[
 r_k
= (\one + \one)^{n + m} \sum_{i=0}^k \frac{(m-k+i)!(m-i)!}{(m-k)!}\frac{(n-k+i)!(n-i)!}{(n-k)!}  p_i q_{k-i}
\]
for all $0 \leq k \leq n$.
\end{corollary}
\begin{proof}
Again, by Lemma~\ref{lem:cp} it is equivalent to show  
\begin{equation}
  \label{eq:needforasymmp}
\sum_{P, P', Q, Q'}  [M(Q, Q', P, P')]^{(k)} 
=  (\one + \one)^{n + m} \sum_{i=0}^k c_{m, k, i}c_{n, k, i} [AC]^{(i)} [BD]^{(k-i)}.
\end{equation}
where 
\[
 c_{m, k, i} = \frac{(m-k+i)!(m-i)!}{(m-k)!}.
\]
Using Lemma~\ref{lem:minors}, we have
\begin{equation}
\label{eq:startasymm}
 [M(Q, Q', P, P')]^{(k)} 
 \refeq{eq:mult} \sum_{X \in \binom{[n]}{k}} \sum_{S \in \binom{[m]}{k}} [Q P A {P'}^{-1} Q' + B]_{X, S} [Q' P' C P^{-1} Q + D]_{S, X}
\end{equation}
where 
\[
 [Q P A {P'}^{-1} Q' + B]_{X, S} = \sum_{i=0}^k \sum_{U, V \in \binom{[k]}{i}} (-\one)^{\| U + V\|_1}  [Q P A {P'}^{-1} Q']_{U(X), V(S)}[B]_{\overline{U}(X), \overline{V}(S)}
\]
and
\[
 [Q' P' C {P}^{-1} Q + D]_{S, X} = \sum_{j=0}^k \sum_{W, Z \in \binom{[k]}{j}} (-\one)^{\| W + Z \|_1}  [Q' P' C {P}^{-1} Q]_{W(S), Z(X)}[D]_{\overline{W}(S), \overline{Z}(X)}.
\]
For fixed $i, j$, we have by Lemma~\ref{lem:asymm}:
\begin{align*}
 \sum_{P, P', Q, Q'} &[Q P A {P'}^{-1} Q']_{U(X), V(S)} [Q' P' C {P}^{-1} Q]_{W(S), Z(X)}
 \\&=  \true{i=j}(\one + \one)^{n+m} i!(n-i)!i!(m-i)! [I_n]_{U(X), Z(X)} [I_m]_{W(S), V(S)} [AC]^{(i)}
 \\&=  \true{i=j}(\one + \one)^{n+m} i!(n-i)!i!(m-i)! [AC]^{(i)} \true{U=Z}\true{W=V}
\end{align*}
and since $(-\one)^{\| W + Z  + U + V \|_1} = \one$ whenever $U = Z$ and $W = V$, we get
\[
 (\ref{eq:startasymm}) = (\one + \one)^{n+m}  \sum_{i=0}^k i!(n-i)!i!(m-i)![AC]^{(i)}   \sum_{X \in \binom{[n]}{k}} \sum_{S \in \binom{[m]}{k}}   \sum_{U, V \in \binom{[k]}{i}}  [D]_{\overline{V}(S), \overline{U}(X)} [B]_{\overline{U}(X), \overline{V}(S)}. 
\]
Similar to Corollary~\ref{cor:symmp}, we can write
\begin{align*}
  \sum_{X \in \binom{[n]}{k}} &\sum_{S \in \binom{[m]}{k}}  \sum_{U, V \in \binom{[k]}{i}}  [D]_{\overline{V}(S), \overline{U}(X)} [B]_{\overline{U}(X), \overline{V}(S)}
  \\&= \sum_{X \in \binom{[n]}{k}} \sum_{S \in \binom{[m]}{k}}  \sum_{W \in \binom{[n]}{k-i}} \sum_{Z \in \binom{[m]}{k-i}}  [D]_{Z, W} [B]_{W, Z} \true{W \subseteq X}\true{Z \subseteq S}
  \\&= \sum_{W \in \binom{[n]}{k-i}} \sum_{Z \in \binom{[m]}{k-i}} \binom{n-k+i}{i}\binom{m-k+i}{i} [D]_{Z, W} [B]_{W, Z} 
   \\&= \binom{n-k+i}{i}\binom{m-k+i}{i} [BD]^{(k-i)} 
\end{align*}
where the last equality used (\ref{eq:mult}) and then (\ref{def:A^k}).
Substituting this gives (\ref{eq:needforasymmp}) as required.
\end{proof}

\section{Conclusion}\label{sec:end}

The results in this paper generalize results from \cite{mypolys, polys} in two ways: firstly, in the types of determinantal constructs that can be considered (more matrices are allowed) and secondly, in the underlying domain of the matrices that can be considered.

One lemma that could has potential for being extended is Lemma~\ref{lem:cancel2}.
We suspect that a formula can be obtained for four sets in general (as was the case for Lemma~\ref{lem:cancel}):
\begin{problem}
 Find a formula for 
\[
 \sum_{P \in \cP_n} [P]_{S, T} [P^{-1}]_{U, V}
\]
that holds for all $S, T, U, V \in \binom{[n]}{k}$.
\end{problem}


\begin{thebibliography}{999}

\bibitem{hornjohnson} Horn, Roger A., and Johnson, Charles R., \textit{Matrix analysis}. Cambridge University Press, 2012.
\bibitem{lang} Serge Lang, \textit{Graduate Texts in Mathematics}, Vol. 211 (2002).

\bibitem{mypolys} Adam W. Marcus, \textit{Polynomial convolutions and (finite) free probability}, preprint.

\bibitem{polys} Adam W. Marcus, Daniel. A Spielman, Nikhil Srivastava,
\textit{Finite free convolutions of polynomials}, preprint. 

\bibitem{IF1} Adam W. Marcus, Daniel. A Spielman, Nikhil Srivastava,
\textit{Interlacing families I: bipartite Ramanujan graphs of all degrees}, 
Ann. of Math. 182-1 (2015), 307--325.

\bibitem{IF2} Adam W. Marcus, Daniel. A Spielman, Nikhil Srivastava,
\textit{Interlacing families II: mixed characteristic polynomials and the Kadison--Singer problem},
Ann. of Math. 182-1 (2015), 327--350.

\bibitem{IF4} Adam W. Marcus, Daniel. A Spielman, Nikhil Srivastava,
\textit{Interlacing families IV: bipartite Ramanujan graphs of all sizes},
Foundations of Computer Science (FOCS), 2015 IEEE 56th Annual Symposium on. IEEE, 2015.




\end{thebibliography}
\end{document}